\documentclass[11pt,reqno,a4paper]{amsart}

\usepackage{hyperref}
\usepackage{amsmath, amssymb}
\usepackage{amscd}
\usepackage{xcolor}
\usepackage{amsthm}
\usepackage[font=footnotesize]{caption}
\usepackage{mathrsfs}

\newtheorem{thm}{Theorem}[section]
\newtheorem{lem}[thm]{Lemma}
\newtheorem{cor}[thm]{Corollary}
\newtheorem{prop}[thm]{Proposition}
\newtheorem{defn}[thm]{Definition}

\theoremstyle{definition}

\newtheorem{rem}[thm]{Remark}

\newcommand{\cT}{\mathcal{T}}
\newcommand{\cF}{\mathcal{F}}
\newcommand{\cG}{\mathcal{G}}
\newcommand{\wt}{\widetilde}
\newcommand{\La}{\Lambda}
\newcommand{\bi}{\bar{i}}
\newcommand{\bistar}{\bar{i^*}}
\newcommand{\ve}{\varepsilon}

\newcommand{\Hom}{\operatorname{Hom}}
\newcommand{\End}{\operatorname{End}}

\newcommand{\GL}{\operatorname{GL}}
\newcommand{\add}{\operatorname{add}}

\newcommand{\soc}{\operatorname{soc}}
\newcommand{\modcat}{\operatorname{mod}}
\newcommand{\modtop}{\operatorname{top}}

\allowdisplaybreaks 
\makeatletter
\@namedef{subjclassname@2020}{%
	\textup{2020} Mathematics Subject Classification}
\makeatother

\begin{document}
\baselineskip=15pt

\title[Torsion pairs and Ringel duality for Schur algebras]{Torsion pairs and Ringel duality for Schur algebras}


\author{Karin Erdmann}
\address[K. Erdmann]{Mathematical Institute, University of Oxford, Radcliffe Observatory Quarter, 
Oxford OX2 6GG, UK}
\email{erdmann@maths.ox.ac.uk}

\author{Stacey Law}
\address[S. Law]{Department of Pure Mathematics and Mathematical Statistics, University of Cambridge, Cambridge CB3 0WB, UK}
\email{swcl2@cam.ac.uk}

\begin{abstract}
	Let $A$ be a finite-dimensional algebra over a field of characteristic $p>0$. We use a functorial approach involving torsion pairs to construct embeddings of endomorphism algebras of basic projective $A$--modules $P$ into those of the torsion submodules of $P$. As an application, we show that blocks of both the classical and quantum Schur algebras $S(2,r)$ and $S_q(2,r)$ are Morita equivalent as quasi-hereditary algebras to their Ringel duals if they contain $2p^k$ simple modules for some $k$.
\end{abstract}



\maketitle

\section{Introduction}

The classical Schur algebras $S(n, r)$ and their $q$-analogues $S_q(n,r)$ are finite-dimensional algebras which arise naturally in algebraic Lie theory.
Viewing them as quasi-hereditary algebras allows tools from the representation theory of finite-dimensional algebras to be used in their study, which has led to new insight, including the discovery of tilting modules in algebraic Lie theory.

The endomorphism algebra of a full tilting module $T$ of a quasi-hereditary algebra is again quasi-hereditary. It is called a Ringel dual of the quasi-hereditary algebra, and is unique up to Morita equivalence (see \cite{R2}).
In \cite{EH}, the main result classifies when a classical Schur algebra $S(2,r)$ is Morita equivalent to its Ringel dual, as an algebra.

In this paper, we study blocks of both classical and quantum Schur algebras from the perspective of torsion pairs. 
More generally, for any finite-dimensional algebra $A$, we show that under suitable conditions the endomorphism algebra of a basic projective $A$--module $P$ embeds into that of the torsion submodule of $P$.
We apply this result to show that a block of $S(2,r)$ or $S_q(2,r)$ with $2p^k$ simple modules is Ringel self-dual. This proof has the advantage of being functorial.

Along the way, we review and extend symmetry properties for decomposition numbers for Schur algebras, as well as multiplicities of standard modules in tilting modules, providing block versions of these results.
A recent theorem by Coulembier \cite{Coul} shows that for a quasi-hereditary algebra with a duality fixing simple modules, there is essentially a unique quasi-hereditary structure. With this, we deduce that the blocks with $2p^k$ simple modules are in fact Ringel self-dual as \emph{quasi-hereditary} algebras. 

The paper is organized as follows: we start with torsion pairs and our result on embeddings of endomorphism algebras described above (Theorem~\ref{thm:2.4}). We then recall preliminary material on quasi-hereditary algebras, and in Section 3 we discuss blocks of $S_q(2,r)$ and $S(2,r)$. In Section 4 we prove a number of combinatorial identities on decomposition numbers and multiplicities for tilting modules, exhibiting symmetry properties in the decomposition matrix and the corresponding matrix for tilting modules. In Section 5 we apply the embedding to give a functorial proof that blocks of Schur algebras with $2p^k$ simple modules are Ringel self-dual as quasi-hereditary algebras.

\bigskip

\subsection*{Acknowledgements}
The second author was supported by a London Mathematical Society Early Career Fellowship at the University of Oxford.

\bigskip

\section{Preliminaries}

Throughout, let $K$ be an algebraically closed field of characteristic $p>0$. 
Assume $A$ is a finite-dimensional $K$--algebra, and denote by $A$-$\modcat$ the category of finite-dimensional left $A$--modules. 

\subsection{Torsion pairs}\label{sec:torsion-prelim}

We follow the notation of \cite[Chapter 4]{ASS}, and refer the reader also to \cite[\textsection 4.1]{R} for further detail.
Recall that  that a module $M$ is \emph{generated} by a module $G$ if it is isomorphic to a quotient of a direct sum of copies of $G$. 

\begin{defn}
	A \emph{torsion pair} is a pair $(\cG,\cF)$ of full subcategories of  $A$-$\modcat$ such that
	\[ \cF = \{M\in A\text{-}\modcat \mid \Hom(\cG,M) = 0 \}\quad\text{and}\quad\cG = \{M\in A\text{-}\modcat \mid \Hom(M,\cF)=0 \}. \]
\end{defn}

Modules in $\cG$ (sometimes also $\cT$ in the literature) and $\cF$ are often referred to as \textit{torsion} and \textit{torsion-free} modules respectively.

\begin{prop}[{\cite[\textsection 4]{R}}]\label{prop:torsion-pair} 
	Let $G$ be a fixed projective $A$--module. Then $(\cG,\cF)$ is a torsion pair, where
	\begin{itemize}
		\item $\cF:=\{ M \mid \Hom(G,M)=0\}$, and
		\item $\cG:=\{M \mid M\ \text{is generated by}\ G \} = \{M \mid \modtop M \in\add(\modtop G) \}$.
	\end{itemize}
\end{prop}

Note here ${\rm add}(V)$ denotes the full subcategory of $A$-$\modcat$ consisting of modules which are direct sums of direct summands of $V$. 

 
\bigskip

\subsection{The algebra map}\label{sec:torsion pair}

Let $G$ be a (basic) projective $A$--module, so $G=Ae$ for some idempotent $e$ of $A$. Defining module categories $(\cG,\cF)$ as in Proposition~\ref{prop:torsion-pair} gives a torsion pair. Furthermore, we have the additive functor 
\[ t:A\text{-}\modcat\longrightarrow A\text{-}\modcat,\qquad M\longmapsto AeM, \]
that is, $t(M)=AeM$ is the largest submodule of $M$ generated by $G$.
For a homomorphism $\theta:M\to N$, the map $t(\theta)$ is the restriction $\theta|_{t(M)}:t(M)\to t(N)$. The following says that any module can be written uniquely as the extension of a torsion-free module by a torsion module.

\begin{prop}[{\cite[VI Proposition 1.5]{ASS}}] \label{prop:ses}
For any $A$--module $M$, there is a short exact sequence
\begin{equation}\label{eqn:t}
0\to t(M) \to M \to \wt{M} \to 0
\end{equation}
where $t(M)\in\cG$ is the largest submodule of $M$ generated by $G$, and $\wt{M}\cong M/t(M)\in\cF$. 
It is unique up to isomorphism of short exact sequences. 
\end{prop}

Let $\{L(\lambda)\mid \lambda\in\Lambda\}$ denote a complete set of pairwise non-isomorphic simple $A$--modules, with corresponding indecomposable projective modules $P(\lambda)$. When clear from context, we write $L\lambda$ and $P\lambda$ for $L(\lambda)$ and $P(\lambda)$, and similarly omit other parentheses, to ease the notation. 

\medskip

\begin{thm}\label{thm:2.4}
	Let $e$ and $t$ be defined as above. Suppose $\Lambda'$ is a subset of $\La$. Assume that for all $\lambda\in\Lambda'$ we have $\soc P(\lambda)= Ae\soc P(\lambda)$.
	Then $t:\End_A(P)\to \End_A(t(P))$ is an injective algebra homomorphism, where	$P=\bigoplus_{\lambda\in\Lambda'}P(\lambda)$.
\end{thm}

Note that the assumption on the socle implies that $t(P\lambda)$ is non-zero. 
Given an arbitrary choice of $G=Ae$, one may take $\Lambda'$ to be the largest subset of $\La$ such that the socle condition holds.
Alternatively,  we may start with a set $\Lambda'\subset \La$ and then take $e$ to be a (smallest) basic idempotent such that
the socle condition holds.

\begin{proof}[Proof of Theorem~\ref{thm:2.4}]
	That the map is an algebra homomorphism is clear. Let $\lambda,\mu\in\Lambda'$. We show that each component
	\begin{equation}\label{eqn:component}
	t_{\lambda\mu}:\Hom_A(P\lambda,P\mu)\to \Hom_A\big( t(P\lambda), t(P\mu) \big)
	\end{equation}
	is injective. That is, for $0\ne\theta\in\Hom_A(P\lambda,P\mu)$, we show that $\theta|_{t(P\lambda)}\ne 0$.
	
 Since $\theta$ is non-zero, there is a simple submodule $S\subseteq  \soc P(\mu)$ such that  $S\subseteq \operatorname{im}\theta$ (note that the socle is semisimple). By the assumption,
 $S= AeS$ and $eS\subseteq S$ is non-zero. Hence there is some $0\neq w' \in \operatorname{im} \theta$ with $w'=ew$ for some $w\in S$. 
 Say $ew=\theta(v)$ for some $v\in P\lambda$. Then $0\ne ew = e^2w = e(\theta(v))=\theta(ev)$, where $ev\in AeP\lambda=t(P\lambda)$.
\end{proof}

We focus on the case $\Lambda'=\Lambda$. If we choose $G$ to be the full basic projective module $\bigoplus_{\lambda\in\Lambda}P(\lambda)$ then the above theorem simply gives the identity map. 
However, choosing $G$ to be a different (basic) projective module gives us a non-trivial embedding of algebras. In our main application, $G$ is the direct sum of all indecomposable projective-injective modules in a block of a Schur algebra, and an application of the combinatorial results in the rest of the paper is the identification of the module $t(P)$ for this choice of $G$. When there are $2p^k$ simples in the block, Theorem~\ref{thm:2.4} allows us to show that the basic algebra of the Schur algebra embeds into its Ringel dual, from which we can deduce Ringel self-duality of the block, in Section~\ref{sec:ringel-self-duality}.

\bigskip

\subsection{Quasi-hereditary algebras and Ringel duality}

Assume $S$ is a finite-dimensional algebra over $K$, and assume $(\La, \leq)$ is a poset labelling the simple $S$--modules. 
For $\lambda\in \La$, let $L(\lambda)$ be the corresponding simple module, and let $P(\lambda)$ and $Q(\lambda)$ be its projective cover and injective hull respectively.
The standard module $\Delta(\lambda)$ is defined to be the largest quotient of $P(\lambda)$ such that all its composition factors are of the form $L(\mu)$ with $\mu\leq \lambda$. Dually, one defines the costandard module $\nabla(\lambda)$. We will also refer to $\lambda$ as the highest weight of the modules just introduced.
With these, the algebra $S$ is \emph{quasi-hereditary} if for all $\lambda\in \La$,
\begin{itemize}
\item[(i)]   the simple module $L(\lambda)$ occurs only once as
a composition factor of $\Delta(\lambda)$, and
\item[(ii)] the projective module $P(\lambda)$ has a filtration with quotients isomorphic to $\Delta(\mu)$ for $\lambda \leq \mu$.
\end{itemize}

Recently, Coulembier showed that a finite-dimensional algebra with simple preserving duality
admits at most one quasi-hereditary structure \cite[Theorem 2.1.1]{Coul}, which we will see in Section~\ref{sec:3.1} is the case for Schur algebras.

\medskip

Throughout, all modules we consider will be finite-dimensional.
For a module $M$ we write $[M:L(\lambda)]$ for the multiplicity of $L(\lambda)$ as a composition factor of $M$.
If $M$ has a filtration whose quotients are standard modules then the multiplicity of $\Delta(\lambda)$ in such a filtration is independent of the filtration (see for example  \cite[\textsection A1 (7)]{D}); we denote this multiplicity by $[M:\Delta(\lambda)]$. 
Similarly, if $M$ has a filtration whose quotients are costandard modules then we write $[M:\nabla(\lambda)]$ for the filtration multiplicity.

If the simple modules have $1$-dimensional endomorphism algebras, which will hold in our setting of Schur algebras from Section~\ref{sec:schur-algs} onwards, then we have the following reciprocity relation:
\[ [P(\lambda):\Delta(\mu)]=[\nabla(\mu):L(\lambda)]\qquad\forall\ \lambda,\mu \in \La. \]
Let $\mathscr{F}(\Delta)$ be the category of $S$--modules which have a filtration by standard modules, and similarly define $\mathscr{F}(\nabla)$.
For each $\lambda\in \La$ there is a unique indecomposable module in $\mathscr{F}(\Delta)\cap \mathscr{F}(\nabla)$ with highest weight $\lambda$, and we denote this module by $T(\lambda)$. We have short exact sequences
\[ 0\to \Delta(\lambda)\to T(\lambda)\to X(\lambda)\to 0\quad \text{and}\quad 0\to Y(\lambda)\to T(\lambda)\to\nabla(\lambda)\to 0 \]
where $X(\lambda)\in\mathscr{F}(\Delta)$ has a standard filtration consisting only of quotients $\Delta(\mu)$ where $\mu\le\lambda$, and similarly $Y(\lambda)\in\mathscr{F}(\nabla)$ has a costandard filtration consisting only of quotients $\nabla(\mu)$ where $\mu\le\lambda$. 


Any direct sum of such indecomposable modules is called a \emph{tilting module}.
Let $T$ be a full tilting module, that is, each $T(\lambda)$ occurs as a summand. Then the endomorphism algebra of $T$ is again quasi-hereditary with respect to $(\La, \leq^{op})$ \cite{R2}. This endomorphism algebra is a \emph{Ringel dual} $S'$ of $S$, and is unique up to Morita equivalence. Moreover, $(S')'$ is Morita equivalent to $S$.

\medskip


\section{Schur algebras $S(2,r)$ and $S_q(2,r)$}\label{sec:schur-algs}

\subsection{The algebras}\label{sec:3.1}

We consider classical Schur algebras $S(2,r)$ and quantum Schur algebras $S_q(2,r)$ for $q$ a fixed primitive $\ell$-th root of unity with $\ell\ge 2$, over $K$.
We refer the reader to \cite{D} and \cite{green} for further details.
One construction is as follows: consider the bialgebra $A(2)$, or $A_q(2)$, with algebra generators $c_{ij}, 1\leq i, j\leq 2$. These are graded with $|c_{ij}|=1$, and can be written as the direct sum  $\bigoplus_{r\geq 0} A(2,r)$, or $\bigoplus_{r\geq 0} A_q(2,r)$, of subcoalgebras, where $A(2,r)$ or 
$A_q(2,r)$ is the component of degree $r$. Then $S(2,r)= A(2,r)^*$ and $S_q(2,r) = A_q(2,r)^*$, each an algebra of dimension ${3+r\choose 3}$.

Let $d_q= c_{11}c_{22} - c_{12}c_{21}$ be the $q$-determinant. Then the localization of $A_q(2)$ at $d_q$ is a Hopf algebra, and $G_q(2)$ is the quantum group whose coordinate algebra is equal to this Hopf algebra. This is 
an appropriate setting for homological techniques analogous to those for group representations to be applied.

The set $\La:= \Lambda^+(2,r)$ of partitions of $r$ with at most two parts labels the simple modules of both $S(2,r)$ and $S_q(2,r)$. For each $\lambda\in\La$, there is a costandard module $\nabla(\lambda)$ with simple socle $L(\lambda)$.
As in the classical case, a costandard module is also isomorphic to a $q$-symmetric power of the natural module, up to tensoring with the $q$-determinant (see \cite[2.1.8]{DD}). There is also a contravariant duality $(-)^{\circ}$ fixing each simple module (see \cite[p.~83]{D}).
The standard module $\Delta(\lambda)$ is defined by $\Delta(\lambda):=\nabla(\lambda)^{\circ}$, and we note that $T(\lambda)^\circ\cong T(\lambda)$ and $P(\lambda)^\circ\cong Q(\lambda)$.

These standard modules and costandard modules define a quasi-hereditary structure on the algebras, and hence their blocks, in both the classical and quantum cases (see \cite{Dhom} and \cite{PW}). The partial order is the dominance order of partitions.

\bigskip

\subsection{Twisted tensor products and short exact sequences}\label{sec:3.2}

The notation for weights as used in \cite{C} is as follows.
Consider a restricted partition $\lambda = (\lambda_1, \lambda_2)$, i.e.~a partition of the form
\[ \lambda = (i + \delta, \delta), \ \ \mbox{where } \  \delta = \lambda_2\ \ \text{and}\ \ i=\lambda_1-\lambda_2 \]
such that $0\leq i \leq \ell-1$ and $\delta\geq 0$.
Let $\rho:= (1,0)$ and $\varpi:=(1,1)$.
We focus on such weights where $0\leq i \leq \ell-2$. Define $\bi$  by $i + \bi  = \ell-2$, and let 
\[ \wt{\lambda} := \bi\rho + (i-\ell+1+\delta)\varpi \]
When $\lambda$ is restricted, note that $\Delta(\lambda)$ and $\Delta(\wt{\lambda})$ are simple. 
In fact, we may parametrise weights $\lambda$ by their difference $i=\lambda_1-\lambda_2$ and $\wt{\lambda}$ will correspond to $\bi=\ell-2-i$, as we now describe, showing why these are particularly useful partitions to consider.

\medskip

We may relate the algebras $S_q(2,r)$ and $S_q(2, r-2)$ for $r\ge 2$ as follows.
Let $\ve = \xi_{(r)}$ (following the notation of \cite[\textsection 4.2]{D}
), a certain idempotent in $S:=S_q(2,r)$. There is an isomorphism of algebras
\begin{equation}\label{eqn:SeS}
S/S\ve S \cong S_q(2, r-2)
\end{equation}
(see \cite[\textsection 4.2 (18)]{D}, and for the classical case \cite[1.2]{E}).

The inflation functor identifies 
the simples, standard and costandard, and tilting  modules of 
$S$ with highest weights $\lambda\neq (r)$ in $\Lambda^+(2, r)$ with those
with highest weight $\lambda - \varpi$ in $\Lambda^+(2, r-2)$.
Iterating this allows us to ignore factors of the $q$-determinant, and therefore in explicit calculations we will identify a weight $\mu$ with $\mu_1-\mu_2$. Note that if $\lambda$ and $\mu$ are partitions of the same size, then $\lambda \geq \mu$ in the dominance order if and only if $\lambda_1-\lambda_2\geq \mu_1-\mu_2$.

\medskip

There is a Frobenius morphism $F: G_q(2)\to G(2)$ \cite[\textsection 1.3]{C}. We also denote the classical Frobenius map $G(2)\to G(2)$ by $F$, and $G_q(2)$ by $G$ and $G(2)=\GL_2(K)$ by $\bar{G}$. 
We recall from \cite{C2} short exact sequences and isomorphisms crucial for computing decomposition numbers and filtration multiplicities. We state these using our convention on the labelling of weights, and write $\bar{L}, \bar{\Delta}, \bar{\nabla}$, $\bar{T}$ for the simple, standard, costandard and tilting modules for the classical Schur algebra when it is useful to distinguish them from modules for the quantum Schur algebra. The tensor products below are over the field $K$, that is, $\otimes=\otimes_K$ throughout.

Note that we only interchange the order in tensor products between classical factors. Within twisted tensor products where one factor is classical but the other is not, we always have the twisted factor to the right, as in \cite{C}.

\begin{prop}[{\cite[\textsection 3]{C2}}]\label{prop:tools} 
	Let $n\in\mathbb{N}$. Let $0\le i\le l-2$ and $\bi$ be such that $i+\bi=l-2$. Then we have non-split short exact sequences 
	\begin{equation}\label{eqn:1}
	0 \longrightarrow \Delta(\bi)\otimes \bar{\Delta}(n-1)^F \longrightarrow \Delta(\ell n+i) \longrightarrow \Delta(i) \otimes \bar{\Delta}(n)^F \longrightarrow 0
	\end{equation}
	and
	\begin{equation}\label{eqn:2}
	0\longrightarrow \Delta(\ell n+i) \longrightarrow T(\ell+i)\otimes\bar{\Delta}(n-1)^F \longrightarrow \Delta\big(\ell(n-1)+\bi\big) \longrightarrow 0.
	\end{equation}
	Moreover, we also have isomorphisms
	\begin{equation}\label{eqn:3}
	\Delta(\ell n-1) \cong \Delta(\ell-1)\otimes\bar{\Delta}(n-1)^F,
	\end{equation}
	\begin{equation}\label{eqn:4}
	T(\ell n-1)\cong T(\ell-1)\otimes\bar{T}(n-1)^F,
	\end{equation}
	\begin{equation}\label{eqn:5} 
	T(\ell n+i)\cong T(\ell+i)\otimes\bar{T}(n-1)^F.
	\end{equation}
	We also have the analogous short exact sequences and isomorphisms for the corresponding classical modules and $0\le i\le p-2$, replacing each occurrence of $\ell$ by $p$ (see \cite{EH}).
\end{prop}

\begin{rem}\label{rem:cartan}
	Recall that the indecomposable tilting modules belong to $\mathscr{F}(\Delta)$; the above can be used inductively to determine their $\Delta$-quotients.
	In particular, the sequence \eqref{eqn:2} immediately gives the $\Delta$-quotients of $T(\ell+i)$. 
	Furthermore, by \eqref{eqn:2} and \eqref{eqn:4}, the total number $\Delta$-quotients of an indecomposable tilting module is a power of $2$, and hence so is each diagonal Cartan number for the Ringel dual of a block of a Schur algebra.
\end{rem}

It is also known precisely which standard modules are irreducible; the classical analogue also holds, replacing $\ell$ by $p$ (see \cite{EH}).
\begin{thm}[{\cite[Corollary 2.1.2]{C}}]\label{thm:simple-delta}
	Let $n\in\mathbb{N}_0$. Then $\Delta(n)$ is irreducible if and only if $n\in\{0,1,\dotsc,l-1\}\cup\{\ell ap^k-1 : a\in\{2,3,\dotsc,p\},\ k\in\mathbb{N}_0\}$.
\end{thm}

There is also a $q$-analogue of Steinberg's tensor product theorem \cite[\textsection 3.2 (5)]{D}.

\medskip

The indecomposable tilting modules, as well as other twisted tensor products that we will consider, have simple tops, by \cite[Lemma 11]{EH} in the classical case and its $q$-analogue which we state below. This may be proved using Frobenius kernels as in the classical case.

\begin{lem}\label{lem:3.3}
	Let $\bar{V}$ be a module for $\bar{G}$, and $\lambda = \ell m + i$ where $0\leq i\leq \ell-2$. Then 
	\[ \Hom_{G}(T(\ell + j)\otimes \bar{V}^F, L(\lambda)) \cong \begin{cases}
	\Hom_{\bar{G}}\big(\bar{V},\bar{L}(m)\big) & \text{if }i+j=\ell-2,\\
	0 & \text{otherwise}.
	\end{cases} \]
\end{lem}

\bigskip

\subsection{Blocks}\label{sec:blocks}

From now on, a block will always refer to a block of a Schur algebra, either classical or quantum. A block will be viewed either as an algebra, or as the set of weights for its simple modules; it will be clear from the context what is meant.
For the rest of the article, unless otherwise stated, we fix an integer $a$ such that $2\leq a\leq p$. Let $B_w$ denote a block with $w$ weights. (In other words, $B$ is a block of size $|B|=w$, meaning that $B$ contains exactly $w$ simple modules.)
Our focus will be on blocks with $ap^k$ weights where $k\in\mathbb{N}_0$, as later we will show in Lemma~\ref{lem:4.6} that blocks $B_w$ with $w$ not of form $ap^k$ cannot be Ringel self-dual.

\medskip

From \cite[\textsection 4.2]{C}, a quantum block is \emph{primitive} if no weight in it is congruent to $-1$ (mod $\ell$). Thus, by \eqref{eqn:1}, primitive blocks of $S_q(2,r)$ are of the form
\[ \{s\le r \mid s=2k\ell+i\ \ \text{or}\ \ s=(2k+1)\ell+\bi,\ \ k\in\mathbb{N}_0 \} \]
where $0\le i\le \ell-2$ and $i+\bi=\ell-2$. The same holds for primitive classical blocks, replacing $\ell$ by $p$.

Moreover, any two classical blocks with the same number of simple modules are Morita equivalent as quasi-hereditary algebras \cite[Theorem 13]{EH}, and using this combined with \cite[Proposition 4.2.4]{C}, so too are any classical block and imprimitive quantum block with the same number of simple modules. By an analogous argument, two primitive quantum blocks with the same number of simple modules are also Morita equivalent as quasi-hereditary algebras, and so it will suffice for our main results to consider blocks $B$ as only dependent on their size, as well as to assume that $B$ is primitive. 

\medskip

In order to prove results for $B=B_{ap^k}$, many of our arguments will involve induction on $k$, and so it will be necessary to describe the above Morita equivalences explicitly as the smaller blocks involved in the induction may themselves be imprimitive. Also intrinsic to the combinatorial patterns found in decomposition numbers and tilting multiplicities is a partitioning of $B_{ap^k}$ into intervals of $p^k$ weights each. We describe these general features in the following section.

\bigskip

\subsection{Intervals}\label{sec:intervals}

Consider the (quantum) primitive block $B$ with lowest weight $i$ where $0\leq i\leq \ell-2$. (The classical case is entirely analogous.) Write $\bi = \ell - 2 - i$. 
Each weight of $B$ is of the form $\lambda=\ell m+i^*$ for some $m\in\mathbb{N}_0$ and $i^*\in\{i,\bi\}$ depending on the parity of $m$.
(Note that we allow the possibility that $p=2$ or $\ell=2$, but of course these cannot occur at the same time since $q$ is a primitive $\ell$-th root of unity in $K$.)

\begin{defn}\label{def:3.7}
	For $c\in\{0,1,\dotsc,p-1\}$ and $k\in\mathbb{N}_0$, define the $p^k$-interval
	\[ I_c^{(k)} := \{ \lambda=\ell m+i^* \in B \mid cp^k\le m<(c+1)p^k \}. \]
\end{defn}
We remark that the above intervals should include an extra subscript $i$ depending on the lowest weight in the block, but we omit this to lighten the notation and interpret $I_c^{(k)}$ to depend on $B$ itself. Letting $B$ be a block such that $|B|=ap^k$, we observe that
\[ B = I_0^{(k)} \sqcup I_1^{(k)} \sqcup \dotsc \sqcup I_{a-1}^{(k)}. \]

\medskip

We now describe the different cases of blocks $B$ that will appear in our inductive arguments. 
\begin{itemize}
	\item[(i)] If $B$ is primitive, then its weights and intervals are as described above. We will without loss of generality always assume that its lowest weight is $i$, for some $0\le i\le \ell-2$ (or $0\le i\le p-2$ in the classical case). 
	
	\item[(ii)] In the inductive steps of our proofs, arbitrary classical blocks will occur, namely when we apply inductive hypotheses on blocks or intervals of smaller size. For a weight $\lambda=\ell m+i^*\in I_c^{(k)}$, we have that $m=cp^k+m_0$ with $0\le m_0\le p^k-1$.
	
	\begin{itemize}
		\item[(a)] If $m_0\not\equiv -1$ (mod $p$), then $m$ lies in a primitive classical block $B_{ap^{k-1}}$. 
		Namely, we can write $m=p(cp^{k-1}+m'_0)+j$ for some $0\le j\le p-2$, and $m$ belongs to the interval $I_c^{(k-1)}$.
		
		\item[(b)] If $m_0\equiv -1$ (mod $p$), then $m$ lies in an imprimitive classical block. Suppose that $m\equiv -1$ (mod $p^t$) but $m\not\equiv -1$ (mod $p^{t+1}$) for some $1\le t<k$. (Note that (a) corresponds to $t=0$, while $t=k$ would hold if and only if $m=(c+1)p^k-1$, in which case $\bar{\Delta}(m)$ is simple.) 
		
		Then we can write $m=(p^t-1) + p^t(cp^{k-t}+m_1)$ where $m_1\not\equiv -1$ (mod $p$). Thus, $m$ belongs to a block $\bar{\Delta}(p^t-1)\otimes \bar{B}^{F^t}$ which is Morita equivalent to $\bar{B}$ as quasi-hereditary algebras (see \cite{EH}), and $\bar{B}$ is a primitive block of size $ap^{k-1-t}$. 
		
		Moreover, $m$ maps to $cp^{k-t}+m_1\in \bar{I}_c^{(k-1-t)}\subseteq \bar{B}$. (Note the lowest weight of $\bar{B}$ is not necessarily equal to $i$ nor $\bi$.)
	\end{itemize}
	
	\item[(iii)] Take two (quantum) blocks $B'$ and $B''$ which both have $0\le i\le\ell-2$ as lowest weight, and write $\Gamma_1$, $\Gamma_2$ for the weights of $B'$, $B''$ respectively. If $|\Gamma_1|\le|\Gamma_2|$, then $\Gamma_1\subseteq \Gamma_2$, and every $B'$--module can be viewed as a $B''$--module. Namely, $B'$ is isomorphic to $B''/B''eB''$ as quasi-hereditary algebras, for some idempotent $e$, which follows from iterating the isomorphism \eqref{eqn:SeS} and projecting onto blocks.
\end{itemize}

\bigskip

\subsection{The socle of $\Delta(\lambda)$}

In this section, let $c\in\{1,2,\dotsc,p-1\}$ and $k\in\mathbb{N}_0$. Let $B=B_{ap^k}$ be a (primitive) block of a Schur algebra. For $0\leq i^*\leq \ell-2$, let $\bistar=\ell-2-i^*$, replacing $\ell$ by $p$ in the classical case. 

\medskip

The socles of standard modules are simple: as remarked in Section~\ref{sec:3.2}, the top of $T(\lambda)$ is simple, so by the contravariant duality $T(\lambda)\cong T(\lambda)^\circ$ we have that $\modtop T(\lambda)=\soc T(\lambda)=\soc\Delta(\lambda)$. 
The following gives an explicit description, and this will be an important input when describing the symmetry of decomposition numbers later.

\begin{defn}\label{def:sigma}
	\begin{itemize}
		\item[(a)] \emph{[Classical weights]}\ Suppose $n=cp^k+n_0$ and $-1\le n_0\le p^k-2$. Define 
		\[ \bar{\sigma}(n):=cp^k-2-n_0. \]
		Note that if $\lambda = pm+i^*$ with $0\leq i^*\leq p-2$ and $m=cp^k+m_0$ with $0\leq m_0\leq p^k-1$, then \[ \bar{\sigma}(\lambda)=p(cp^k-m_0-1)+\bistar=p\bar{\sigma}(m-1) + \bistar. \]
		
		\item[(b)] \emph{[Quantum weights]}\ Suppose $\lambda=\ell m+i^*$ where $0\le i^*\le \ell-2$ and $m=cp^k+m_0$ with $0\le m_0\le p^k-1$. Define 
		\[ \sigma(\lambda)=\ell(cp^k-m_0-1)+\bistar=\ell\bar{\sigma}(m-1) + \bistar. \]
	\end{itemize}
\end{defn}

The following is clear by Definition~\ref{def:sigma}; we remark that the classical analogue also holds.

\begin{lem}\label{lem:3.9} 
	The map $\sigma$ gives an order-reversing bijection $I_c^{(k)} \to I_{c-1}^{(k)}$ for all $1\le c\le p-1$. In particular, every weight $\mu\in I_d^{(k)}$ for $0\le d\le p-2$ is of the form $\mu=\sigma(\lambda)$ for a unique  $\lambda\in I_{d+1}^{(k)}$.
\end{lem}

\begin{lem}[{\cite[Lemma 3]{EH}}]\label{lem:c-socle}
	Suppose $n=cp^k+n_0$ and $-1\le n_0\le p^k-2$. Then $\soc\bar{\Delta}(n) = \bar{L}(\bar{\sigma}(n))$. 
\end{lem}


\begin{lem}\label{lem:q-socle}
	\begin{itemize}
		\item[(a)] Suppose $\lambda = pm+i^*$ with $0\leq i^*\leq p-2$ and
			$m=cp^k+m_0$ with $0\leq m_0\leq p^k-1$. 
			Then $\soc\bar{\Delta}(\lambda)=\bar{L}(\bar{\sigma}\lambda)$. 
			
		\item[(b)] Suppose $\lambda=\ell m+i^*$ where $0\le i^*\le \ell-2$ and $m=cp^k+m_0$ with $0\le m_0+p^k-1$. Then $\soc\Delta(\lambda)=L(\sigma\lambda)$. 
	\end{itemize}
\end{lem}

\begin{proof}
Part (a) follows from Lemma~\ref{lem:c-socle}. For part (b), we have by \eqref{eqn:1} that the socle of $\Delta(\lambda)$ equals that of $\bar{\Delta}(m-1)^F\otimes L(\bistar)$, which is $L(\ell\bar{\sigma}(m-1)+\bistar)$. 
\end{proof}


We make two more observations useful for our investigations into decomposition numbers and tilting multiplicities later.

\begin{lem}\label{lem:3.10} 
	Let $\rho$ be the largest weight in $I_c^{(k)}$ where $1\leq c\leq a-1$. Then $L(\sigma\rho)$ is the unique composition factor of $\Delta(\rho)$ with highest weight in $I_{c-1}^{(k)}$.
\end{lem}

\begin{proof}
We have $\rho = \ell s+i^*$ with $s=(c+1)p^k-1$. From \eqref{eqn:1}, we have the exact sequence
\[ 0\to L(\bar{i^*})\otimes \bar{\Delta}(s-1)^F\to \Delta(\rho) \to L(i^*)\otimes \bar{\Delta}(s)^F\to 0. \]
The last term of the sequence is the simple module $L(\rho)$, by Steinberg's tensor product theorem. Letting $\mu_t = (c+1)p^t-2$ for $0\leq t\leq k$, then the first term of the exact sequence
is $L(\bar{i^*}) \otimes \bar{\Delta}(\mu_k)^F$.

Moreover, for any $t\ge 1$ the classical version of \eqref{eqn:1} gives the exact sequence
\[ 0\to \bar{\Delta}(\mu_{t-1})^F  \to \bar{\Delta}(\mu_t) \to \bar{L}(p-2)\otimes \bar{\Delta}((c+1)p^{t-1}-1)^F \to 0 \]
where the last term is isomorphic to $\bar{L}(\mu_t)$. 
Furthermore, we have $\bar{\Delta}(\mu_0) = \bar{L}(c-1)$. 
This shows that the composition factors of $\bar{\Delta}(\mu_k)$ have highest weights
\[ \mu_k, \ \ p\mu_{k-1}, \ \  p^2\mu_{k-2}, \ \ \ldots, \ \ p^k\mu_0. \]
Therefore the composition factors of $L(\bar{i^*})\otimes \bar{\Delta}(s-1)^F$ have highest weights
$\ell p^t\mu_{k-t} + \bar{i^*}$ for $0\leq t\leq k$. 

Now observe that $p^t\mu_{k-t} = (c+1)p^k - 2p^t = cp^k+ (p^k-2p^t)$, and $p^k-2p^t\geq 0$ if and only if $t\leq k-1$. Also, $p^k\mu_0 = p^k(c-1)$. 
Hence $\ell p^t\mu_{k-t} + \bar{i^*}$ belongs 
to $I_c^{(k)}$ if $t\leq k-1$, and to $I_{c-1}^{(k)}$ if $t=k$. 
This shows that $\sigma(\rho)=\ell p^k\mu_0 + \bar{i^*}$ is the only highest weight of a composition factor of $\Delta(\rho)$ which is in $I_{c-1}^{(k)}$.
\end{proof}

\begin{lem}\label{lem:3.11}
	Let $\lambda$ be in $B_{ap^k}$, and assume $\lambda\not\in I_0^{(k)}$. Then $T(\lambda)$ has precisely two $\Delta$-quotients if and only if $\lambda$ is the smallest weight in $I_c^{(k)}$ for some $c\geq 1$.
\end{lem}

\begin{proof}
	Let $\lambda=\ell m+i^*$. By \eqref{eqn:5}, we have that $T(\lambda)\cong T(\ell+i^*)\otimes\bar{T}(m-1)^F$.
	By \eqref{eqn:2}, modules of the form $T(\ell+i^*)\otimes\bar{\Delta}(n)^F$ each have two $\Delta$-quotients. Thus if $T(\lambda)$ has precisely two $\Delta$-quotients, then $\bar{T}(m-1)$ has precisely one $\bar{\Delta}$-quotient and so is simple. Therefore $m-1=cp^k-1$ for some $1\le c\le p-1$, and hence $\lambda=\ell cp^k+i^*$ which is the smallest weight in $I_c^{(k)}$. The converse is clear.
\end{proof}

\bigskip

\section{The decomposition matrix and the tilting matrix of a block}\label{sec:4}

Throughout Section~\ref{sec:4}, fix $B$ a primitive block of a Schur algebra of size $ap^k$, where $a\in\{2,3,\dotsc,p\}$ and $k\in\mathbb{N}_0$. Let its lowest weight be $i\in\{0,1,\dotsc,\ell-2\}$ and let $\bistar=\ell-2-i^*$, replacing $\ell$ by $p$ in the classical case. Let $c\in\{1,2,\dotsc,a-1\}$.

\medskip

The decomposition matrix of $B$ is defined to be the $ap^k\times ap^k$ matrix indexed by weights $\lambda,\mu\in B$ whose $(\lambda,\mu)$-entry is the decomposition number $[\Delta(\lambda):L(\mu)]$. Here we order the weights $\lambda,\mu$ 
by the natural order on integers. By the tilting matrix of $B$, we mean the $ap^k\times ap^k$ matrix indexed in the same way whose $(\lambda,\mu)$-entry is $[T(\lambda):\Delta(\mu)]$. We refer to these quantities as tilting multiplicities.

While the classical decomposition numbers $[\bar{\Delta}(\lambda):\bar{L}(\mu)]$ are known in terms of the $p$-adic expansions of $\lambda$ and $\mu$ (see \cite[Theorem 2.1]{H}), 
their values can be calculated by straightforward induction on the short exact sequences and isomorphisms introduced in Section~\ref{sec:3.2}. The same is true of the tilting multiplicities 
(see \cite[Lemma 6]{EH} and \cite[\textsection 3.4 (3)]{D}), but we remark that in proving the following propositions we do not need to use the explicit formulas: it will turn out to be natural to prove these combinatorial patterns by induction using only \eqref{eqn:1}--\eqref{eqn:5}, which is the approach we take. 

We illustrate these patterns for $p=5$ in Figures~\ref{fig:decomp} and~\ref{fig:tilting} below.
In both figures, the part of the matrix shown corresponds to those rows and columns indexed by $I_0^{(2)}$, and the five subdivisions indicate the intervals $I_0^{(1)},\dotsc, I_4^{(1)}$. As evident in these figures, we remark that $[\Delta(\lambda):L(\mu)]$ and $[T(\lambda):\Delta(\mu)]\in\{0,1\}$. 

We give an example of such inductive arguments using the short exact sequences in the proof of the following result.

\begin{figure}[h]
	\begin{tiny}
		\setlength\arraycolsep{2pt}
		\[
		\begin{array}{r|*{5}c|*{5}c|*{5}c|*{5}c|*{5}c}
		\hline
		0& 1 &&&& &&&&& &&&&& &&&&& &&&&&\\
		1& 1&1 &&& &&&&& &&&&& &&&&& &&&&&\\
		2& .&1&1 && &&&&& &&&&& &&&&& &&&&&\\
		3& .&.&1&1 & &&&&& &&&&& &&&&& &&&&&\\
		4& .&.&.&1&1 &&&&& &&&&& &&&&& &&&&&\\
		\hline
		5& .&.&.&1&1& 1&&&& &&&&& &&&&& &&&&&\\
		6& .&.&1&1&.& 1&1&&& &&&&& &&&&& &&&&&\\
		7& .&1&1&.&.& .&1&1&& &&&&& &&&&& &&&&&\\
		8& 1&1&.&.&.& .&.&1&1& &&&&& &&&&& &&&&&\\
		9& 1&.&.&.&.& .&.&.&1&1 &&&&& &&&&& &&&&&\\
		\hline
		10& .&.&.&.&.& .&.&.&1&1& 1&&&& &&&&& &&&&&\\
		11& .&.&.&.&.& .&.&1&1&.& 1&1&&& &&&&& &&&&&\\
		12& .&.&.&.&.& .&1&1&.&.& .&1&1&& &&&&& &&&&&\\
		13& .&.&.&.&.& 1&1&.&.&.& .&.&1&1& &&&&& &&&&&\\
		14& .&.&.&.&.& 1&.&.&.&.& .&.&.&1&1 &&&&& &&&&&\\
		\hline
		15& .&.&.&.&.& .&.&.&.&.& .&.&.&1&1& 1&&&& &&&&&\\
		16& .&.&.&.&.& .&.&.&.&.& .&.&1&1&.& 1&1&&& &&&&&\\
		17& .&.&.&.&.& .&.&.&.&.& .&1&1&.&.& .&1&1&& &&&&&\\
		18& .&.&.&.&.& .&.&.&.&.& 1&1&.&.&.& .&.&1&1& &&&&&\\
		19& .&.&.&.&.& .&.&.&.&.& 1&.&.&.&.& .&.&.&1&1 &&&&&\\
		\hline
		20& .&.&.&.&.& .&.&.&.&.& .&.&.&.&.& .&.&.&1&1& 1&&&&\\
		21& .&.&.&.&.& .&.&.&.&.& .&.&.&.&.& .&.&1&1&.& 1&1&&&\\
		22& .&.&.&.&.& .&.&.&.&.& .&.&.&.&.& .&1&1&.&.& .&1&1&&\\
		23& .&.&.&.&.& .&.&.&.&.& .&.&.&.&.& 1&1&.&.&.& .&.&1&1&\\
		24& .&.&.&.&.& .&.&.&.&.& .&.&.&.&.& 1&.&.&.&.& .&.&.&1&1\\
		\end{array} \]
		\begin{center}
			$\vdots$
		\end{center}
	\end{tiny}
	\caption{Decomposition matrix for $p=5$. The $(m,n)$-entry denotes $[\Delta(\lambda):L(\mu)]$ where $\lambda=\ell m+i_1^*$, $\mu=\ell n+i_2^*$ in the quantum case, or $[\bar{\Delta}(\lambda):\bar{L}(\mu)]$ where $\lambda=pm+i_1^*$, $\mu=pn+i_2^*$ in the classical case.}\label{fig:decomp}
%
\vspace{40pt}
	\begin{tiny}
		\setlength\arraycolsep{2pt}
		\[
		\begin{array}{r|*{5}c|*{5}c|*{5}c|*{5}c|*{5}c}
		\hline
		0& 1 &&&& &&&&& &&&&& &&&&& &&&&&\\
		1& 1&1 &&& &&&&& &&&&& &&&&& &&&&&\\
		2& .&1&1 && &&&&& &&&&& &&&&& &&&&&\\
		3& .&.&1&1 & &&&&& &&&&& &&&&& &&&&&\\
		4& .&.&.&1&1 &&&&& &&&&& &&&&& &&&&&\\
		\hline
		5& .&.&.&.&1&1 &&&& &&&&& &&&&& &&&&&\\
		6& .&.&.&1&1&1&1 &&& &&&&& &&&&& &&&&&\\
		7& .&.&1&1&.&.&1&1 && &&&&& &&&&& &&&&&\\
		8& .&1&1&.&.&.&.&1&1 & &&&&& &&&&& &&&&&\\
		9& 1&1&.&.&.&.&.&.&1&1 &&&&& &&&&& &&&&&\\
		\hline
		10& .&.&.&.&.& .&.&.&.&1&1 &&&& &&&&& &&&&&\\
		11& .&.&.&.&.& .&.&.&1&1&1&1 &&& &&&&& &&&&&\\
		12& .&.&.&.&.& .&.&1&1&.&.&1&1 && &&&&& &&&&&\\
		13& .&.&.&.&.& .&1&1&.&.&.&.&1&1 & &&&&& &&&&&\\
		14& .&.&.&.&.& 1&1&.&.&.&.&.&.&1&1 &&&&& &&&&&\\
		\hline
		15& .&.&.&.&.& .&.&.&.&.& .&.&.&.&1&1 &&&& &&&&&\\
		16& .&.&.&.&.& .&.&.&.&.& .&.&.&1&1&1&1 &&& &&&&&\\
		17& .&.&.&.&.& .&.&.&.&.& .&.&1&1&.&.&1&1 && &&&&&\\
		18& .&.&.&.&.& .&.&.&.&.& .&1&1&.&.&.&.&1&1 & &&&&&\\
		19& .&.&.&.&.& .&.&.&.&.& 1&1&.&.&.&.&.&.&1&1 &&&&&\\
		\hline
		20& .&.&.&.&.& .&.&.&.&.& .&.&.&.&.& .&.&.&.&1&1 &&&&\\
		21& .&.&.&.&.& .&.&.&.&.& .&.&.&.&.& .&.&.&1&1&1&1 &&&\\
		22& .&.&.&.&.& .&.&.&.&.& .&.&.&.&.& .&.&1&1&.&.&1&1 &&\\
		23& .&.&.&.&.& .&.&.&.&.& .&.&.&.&.& .&1&1&.&.&.&.&1&1 &\\
		24& .&.&.&.&.& .&.&.&.&.& .&.&.&.&.& 1&1&.&.&.&.&.&.&1&1\\
		\end{array} \]
		\begin{center}
			$\vdots$
		\end{center}
	\end{tiny}
	\caption{Tilting matrix for $p=5$. The $(m,n)$-entry denotes $[T(\lambda):\Delta(\mu)]$ where $\lambda=\ell m+i_1^*$, $\mu=\ell n+i_2^*$ in the quantum case, or $[\bar{T}(\lambda):\bar{\Delta}(\mu)]$ where $\lambda=pm+i_1^*$, $\mu=pn+i_2^*$ in the classical case.}\label{fig:tilting}
\end{figure}

\begin{lem}\label{lem:interval-c}
	Let $\rho\in I_c^{(k)}$ and $\gamma\in B$. If $[\Delta(\rho):L(\gamma)]>0$, then $\gamma\in I_{c-1}^{(k)}\sqcup I_c^{(k)}$.
\end{lem}

\begin{proof}
	We proceed by induction on $k$. When $k=0$, the claim follows immediately from \eqref{eqn:1} and Theorem~\ref{thm:simple-delta} since $c\ge 1$. Now suppose $k\ge 1$. For the inductive step, suppose $\rho=\ell n+i^*\in I_c^{(k)}$, so $cp^k\le n<(c+1)p^k$. By \eqref{eqn:1}, the composition factors of $\Delta(\rho)$ are precisely those of $\bar{\Delta}(n-1)^F\otimes L(\bistar)$ and $\bar{\Delta}(n)^F\otimes L(i^*)$, namely
	\[ \begin{array}{llll}
	L(\ell u+\bistar) & \text{s.t.} & [\bar{\Delta}(n-1):\bar{L}(u)]>0, & \text{and}\\
	L(\ell v+i^*) & \text{s.t.} & [\bar{\Delta}(n):\bar{L}(v)]>0. &
	\end{array} \]
	First consider $L(\ell v+i^*)$. If $n\ne (c+1)p^k-1$, then as described in Section~\ref{sec:intervals} (ii), $n=p^t-1+p^t(n_1+cp^{k-t})$ for some $0\le t<k$ with $n_1+cp^{k-t}\in \bar{I}_c^{(x)}$ where $x=k-1-t$. By the inductive hypothesis, $[\bar{\Delta}(n):\bar{L}(v)]>0$ implies that $v=p^t-1+p^tw$ for some $w\in \bar{I}_{c-1}^{(x)}\sqcup \bar{I}_c^{(x)}$. Hence $\ell v+i^*\in I_{c-1}^{(k)}\sqcup I_c^{(k)}$ as $(c-1)p^k\le v<(c+1)p^k$.
	
	If $n=(c+1)p^k-1$ then $\bar{\sigma}(n)=n$, so $\bar{\Delta}(n)$ is simple by Lemma~\ref{lem:q-socle}. Thus $[\bar{\Delta}(n):\bar{L}(v)]>0$ implies $v=n$, and so $\ell v+i^*=\rho\in I_c^{(k)}$.
	
	\medskip
	
	Now consider $L(\ell u+\bistar)$. If $n\ne cp^k$ then $\ell(n-1)+\bistar\in I_c^{(k)}$ is not the maximal weight in this interval, so by a similar argument to the above, $[\bar{\Delta}(n-1):\bar{L}(u)]>0$ implies $\ell u+\bistar\in I_{c-1}^{(k)}\sqcup I_c^{(k)}$.
	If $n=cp^k$, then $\bar{\sigma}(n-1)=n-1$, so $\bar{\Delta}(n-1)$ is simple by Lemma~\ref{lem:q-socle}. Thus $[\bar{\Delta}(n-1):\bar{L}(u)]>0$ implies $u=n-1$, and so $\ell u+\bistar=\ell(cp^k-1)+\bistar\in I_{c-1}^{(k)}$.
\end{proof}

We now proceed with the main results of this section, which relate the decomposition numbers of $B$ to the tilting multiplicities of $B$.

\begin{prop}\label{prop:4.1}
	Assume $\lambda\in I_c^{(k)}$. Let $\rho$ be a weight in $B$. Then
	\[ [T(\lambda):\Delta(\rho)] = [\Delta(\rho):L(\sigma\lambda)] . \]
\end{prop}

The proof of the proposition is postponed to Section~\ref{sec:proof4.1}. We first describe some properties of decomposition numbers and tilting multiplicities that we can deduce from this.

\begin{rem}\label{rem:4.2}
	\begin{itemize}
		\item[(a)] If this multiplicity is positive, then $\sigma\lambda\le\rho\le\lambda$.
		\item[(b)] Moreover, Proposition~\ref{prop:4.1} shows that the column of the decomposition matrix corresponding to $L(\sigma\lambda)$ is the same as the row in the tilting matrix corresponding to $T(\lambda)$. In particular, the number of non-zero entries in the column for $L(\sigma\lambda)$ is a power of $2$, since it is the total number of $\Delta$-quotients of $T(\lambda)$. 
	\end{itemize}
\end{rem}

\begin{cor}\label{cor:4.3}
	If $\lambda\in I_c^{(k)}$, then $T(\lambda)\cong P(\sigma\lambda)$, the indecomposable projective $B$--module with simple top $L(\sigma\lambda)$.
\end{cor}

\begin{proof}
	Note $\modtop T(\lambda)=\soc\Delta(\lambda) = L(\sigma\lambda)$, so there exists a surjective map $\pi:P(\sigma\lambda)\to T(\lambda)$. 
	But $[\Delta(\rho):L(\sigma\lambda)] = [P(\sigma\lambda):\Delta(\rho)]$ by reciprocity, so $P(\sigma\lambda)$ and $T(\lambda)$ have the same $\Delta$-quotients by Proposition~\ref{prop:4.1}. Hence $\pi$ is an isomorphism. 
\end{proof}

This corollary shows that $P(\mu)$ has simple socle $L(\mu)$ whenever $\mu\in B\setminus I_{a-1}$. Namely, if $\mu\in I_d$ with $0\le d\le a-1$, then $\mu = \sigma\lambda$ for $\lambda$ in $I_{d+1}$, and notice $T(\lambda)$ has simple socle and top isomorphic to that of $\Delta(\lambda)$, which is $L(\sigma\lambda)$ by Lemma~\ref{lem:q-socle}.

\begin{prop}\label{prop:4.4}
	Let $\mu \in I_{c-1}^{(k)}$ and let $\lambda,\rho\in I_c^{(k)}$. Then
	\begin{itemize}
		\item[(a)] $[\Delta(\rho):L(\mu)] =  [\Delta(\sigma\rho):L(\mu)]$, and
		\item[(b)] $[T(\sigma\lambda):\Delta(\sigma\rho)] = [\Delta(\rho):L(\lambda)]$.
	\end{itemize}
\end{prop}

\begin{rem} \label{rem:4.5}
	Proposition~\ref{prop:4.4}(a) shows that the entries in the column corresponding to $L(\mu)$ in the decomposition matrix
	are symmetric about the horizontal line between $I_{c-1}^{(k)}$ and $I_c^{(k)}$ (see Figure~\ref{fig:decomp}, for example).
	Proposition~\ref{prop:4.4}(b) shows that $\sigma$ is a bijection from the part of the decomposition matrix indexed by weights in $I_c^{(k)}$ to the part of the tilting matrix indexed by weights in $I_{c-1}^{(k)}$ (see Figures~\ref{fig:decomp} and~\ref{fig:tilting}).
\end{rem}

The proof of Proposition~\ref{prop:4.4} is postponed to Section~\ref{sec:proof4.4}. We first use the results introduced thus far to show that blocks $B$ with $|B|$ not of the form $ap^k$ always have at least one diagonal Cartan number which is not a power of 2. By Remark~\ref{rem:cartan}, such a block therefore cannot be Ringel self-dual, as each diagonal Cartan number for the Ringel dual of a block is a power of $2$.

\begin{lem}\label{lem:4.6}
	Assume $\hat{B}$ is a block with $|\hat{B}|=ap^k+s$ where $k>0$ and $1\le s<p^k$. Then some diagonal Cartan number of $\hat{B}$ is not a power of 2.
\end{lem}

\begin{proof}
	Observe that $B_{ap^k}\subset \hat{B}\subset \wt{B}$ where $\wt{B} = B_{(a+1)p^k}$ if $a<p$, and $\wt{B}= B_{2p^{k+1}}$ otherwise.
	
	Let $\lambda \in \wt{B}$ be the smallest weight which is not in $\hat{B}$. Note that $\lambda\in I_{a}^{(k)}$ (respectively $\lambda\in I_1^{(k+1)}$). Then $\mu:=\sigma\lambda\in I_{a-1}^{(k)}$ (respectively $\in I_0^{(k+1)}$) and hence $\mu\in\hat{B}$. We claim that the diagonal Cartan number $c_{\mu,\mu}=[P(\mu):L(\mu)]$ corresponding to the weight $\mu\in \hat{B}$ is not a power of $2$.
	
	By Remark \ref{rem:4.2}, in the decomposition matrix of $\wt{B}$ in the column corresponding to $L(\mu)$, all entries equal to 1 lie between the rows indexed by $\mu$ and $\lambda$, and the total number of 1s is a power of 2, say $2^t$. It follows that in the decomposition matrix of $\hat{B}$, the number of 1s in the column corresponding to $L(\mu)$ it $2^t-1$, since 
	$\soc\Delta(\lambda)=L(\sigma\lambda)$ and so $[\Delta(\lambda):L(\mu)]=1$.
	
	Assume 
	that all diagonal Cartan numbers of $\hat{B}$ are powers of 2, and so that $2^t-1$ is a power of 2. Thus $t=1$. From Proposition~\ref{prop:4.1}, $[\Delta(\rho):L(\sigma\lambda)]=[T(\lambda):\Delta(\rho)]$ for all $\rho\in\wt{B}$, 
	hence $T(\lambda)$ has precisely two $\Delta$-quotients. By Lemma~\ref{lem:3.11}, $\lambda$ must be the smallest weight in $I_a^{(k)}$ (respectively $I_1^{(k+1)}$). But this implies that $\hat{B}=B_{ap^k}$, a contradiction.
\end{proof}

For the proofs of Propositions~\ref{prop:4.1} and~\ref{prop:4.4}, we present the arguments for a quantum block; for classical blocks it is the same, replacing $\ell$ by $p$. We proceed by induction on $k$, starting with $k=0$.

\bigskip

\subsection{The block $B$ of size $a\le p$}\label{sec:k=0}
We consider the case $k=0$ of Propositions~\ref{prop:4.1} and~\ref{prop:4.4}.

The weights in $B$ are $\{ i, \ell+\bi, 2\ell + i, \ldots, (a-1)\ell +i^*\}$, where $i^*=i$ or $\bi$ depending on the parity of $a$.
The decomposition matrix of $B$ is the top left $a\times a$ submatrix of the following $p\times p$ matrix
\[\begin{bmatrix} 
1 & 0 & 0 & \ldots & 0 \\
1 & 1 & 0 & \ldots & 0 \\
0 & 1 & 1 & \ldots & 0 \\
& \vdots && \vdots & \\
0 & \ldots & 0 & 1 & 1
\end{bmatrix}\]
and the tilting matrix has the same form.  We observe that Propositions~\ref{prop:4.1} and~\ref{prop:4.4} clearly hold when $k=0$, since:
\begin{itemize}
	\item $I_c^{(0)}=\{\ell c + i^*\}$, where $i^*$ depends on the parity of $c$; and
	\item for $c\ge 1$, we have that $\sigma(\ell c + i^*)=\ell(c-1)+\bistar$.
\end{itemize}

\bigskip

\subsection{Proof of Proposition~\ref{prop:4.1}}\label{sec:proof4.1}
We now present the inductive step of the proof of Proposition~\ref{prop:4.1}: let $k\ge 1$ and assume that Proposition~\ref{prop:4.1} is true for all $x$ such that $0\le x<k$. (In fact, we only need the inductive hypothesis for classical blocks.) We want to prove
\[ [\Delta(\rho):L(\sigma\lambda)] = [T(\lambda):\Delta(\rho)] \]
for all $\lambda=\ell m+i^*\in I_c^{(k)}$ with $cp^k\le m<(c+1)p^k$, and all $\rho\in B$.

\begin{proof}
	Suppose $\lambda=\ell m+i^*$ with $m=cp^k+m_0$ and $0\le m_0\le p^k-1$. By Definition~\ref{def:sigma}, we have that
	\[ \sigma\lambda = \ell(cp^k-m_0-1)+\bistar = \ell\bar{\sigma}(m-1)+\bistar. \]
	By \eqref{eqn:5}, $T(\lambda)\cong T(\ell+i^*)\otimes\bar{T}(m-1)^F$, which has a filtration with quotients given by $T(\ell+i^*)\otimes\bar{\Delta}(u)^F$ where $\bar{\Delta}(u)$ runs over the quotients of a $\bar{\Delta}$-filtration of $\bar{T}(m-1)$. 
	By \eqref{eqn:2}, $T(\ell+i^*)\otimes\bar{\Delta}(u)^F$ has $\Delta$-quotients $\Delta\big(\ell(u+1)+i^*\big)$ and $\Delta(\ell u+\bistar)$.
	
	Write $\rho\in B$ as $\rho\in\{\ell(v+1)+i^*,\ell v+\bistar \}$ for some $v$. Then
	\begin{align*}
	[T(\lambda):\Delta(\rho)] &= [\bar{T}(m-1):\bar{\Delta}(u)]\\
	&= [\bar{\Delta}(v):\bar{L}\big(\bar{\sigma}(m-1)\big)]\\
	&= [\bar{\Delta}(v)^F\otimes L(\bistar):\bar{L}\big(\bar{\sigma}(m-1)\big)^F\otimes L(\bistar)]\\
	&= [\Delta(\rho):L(\sigma\lambda)],
	\end{align*}
	where the second equality follows from the inductive hypothesis, and the final equality follows from \eqref{eqn:1} and observing that $L(\sigma\lambda)=\bar{L}\big(\bar{\sigma}(m-1)\big)^F\otimes L(\bistar)$.
\end{proof}

\bigskip

\subsection{Proof of Proposition~\ref{prop:4.4}}\label{sec:proof4.4}
We now present the inductive step of the proof of Proposition~\ref{prop:4.4}: let $k\ge 1$ and assume that Proposition~\ref{prop:4.4} is true for $x$ such that $0\le x<k$. Let $\mu\in I_{c-1}^{(k)}$.

\bigskip

\noindent (a) We want to prove
\[ [\Delta(\rho):L(\sigma\lambda)] = [\Delta(\sigma\rho):L(\sigma\lambda)] \]
for all $\rho\in I_c^{(k)}$ and $\mu\in I_{c-1}^{(k)}$.

\begin{proof}
	We may assume without loss of generality that $\mu=\ell m+i$ where $m=(c-1)p^k+m_0$ and $0\le m_0\le p^k-1$ (the argument is identical if $\mu=\ell m+\bi$). Let $\rho=\ell s+i^*$ where $i^*\in\{i,\bi\}$ and $s=cp^k+s_0$ for some $0\le s_0\le p^k-1$.
	
	First suppose that $\rho$ is the largest weight in $I_c^{(k)}$, that is, $s_0=p^k-1$. Then $\sigma\rho$ is the smallest weight in $I_{c-1}^{(k)}$. Hence 
	\[ [\Delta(\sigma\rho):L(\mu)] = \begin{cases}
	1 & \text{if }\sigma\rho=\mu,\\
	0 & \text{otherwise}.
	\end{cases}1\]
	By Lemma~\ref{lem:3.10}, this is equal to $[\Delta(\rho):L(\mu)]$.
	
	From now on, we may assume $s_0<p^k-1$. We have $L(\mu)\cong \bar{L}(m)^F\otimes L(i)$ by Steinberg's tensor product theorem. By \eqref{eqn:1}, $\Delta(\rho)$ has a filtration with quotients $\bar{\Delta}(s)^F\otimes L(i^*)$ and $\bar{\Delta}(s-1)^F\otimes L(\bistar)$. Hence
	\[ [\Delta(\rho):L(\mu)] = \begin{cases}
	[\bar{\Delta}(s):\bar{L}(m)] & \text{if } i^*=i,\\
	[\bar{\Delta}(s-1):\bar{L}(m)] & \text{if } i^*=\bi.
	\end{cases} \]
	When $i\not\equiv \bi$ (mod $\ell$), this follows from Steinberg and considering the residues modulo $\ell$ of the highest weights involved. If $i\equiv\bi$ (mod $\ell$), then we take $[\bar{\Delta}(s):\bar{L}(m)]$ if $m$ and $s$ have the same parity, and $[\bar{\Delta}(s-1):\bar{L}(m)]$ otherwise, since weights in the same block must have the same parity. A similar remark applies in all cases below where cases arise depending on $i$ or $\bi$, and henceforth we will not distinguish whether $i\equiv\bi$ or $i\not\equiv\bi$ (mod $\ell$).
	
	On the other hand, we have $\sigma\rho=\ell(cp^k-s_0-1)+\bistar$. Observing that $cp^k-s_0-1=\bar{\sigma}(s-1)$ and that $cp^k-s_0-2=\bar{\sigma}(s)$ as $s_0<p^k-1$, we similarly have that
	\[ [\Delta(\sigma\rho):L(\mu)] = \begin{cases}
	[\bar{\Delta}\big(\bar{\sigma}(s-1)\big):\bar{L}(m)] & \text{if } \bistar=i,\\
	[\bar{\Delta}(\bar{\sigma}s):\bar{L}(m)] & \text{if } \bistar=\bi.
	\end{cases} \]
	The inductive hypothesis now directly implies our claim.
\end{proof}

\bigskip

\noindent (b) We want to prove
\[ [T(\sigma\lambda):\Delta(\sigma\rho)] = [\Delta(\rho):L(\lambda)] \]
for all $\lambda,\rho\in I_c^{(k)}$. We may without loss of generality assume that $\lambda=\ell m+i$; the argument for $\lambda=\ell m+\bi$ is identical.

\begin{proof}
	First assume that $\lambda$ is the largest weight in $I_c^{(k)}$. That is, $\lambda=\ell m+i$ where $m=cp^k+p^k-1$. Since $\rho\le\lambda$, we have that
	\[ [\Delta(\rho):L(\lambda)] = \begin{cases}
	1 & \text{if }\rho=\lambda,\\
	0 & \text{otherwise}.
	\end{cases} \]
	By Lemma~\ref{lem:3.9}, $\sigma\lambda$ is the smallest weight in $I_{c-1}^{(k)}$ and $\sigma\lambda\le\sigma\rho$. Thus 
	\[ [T(\sigma\lambda):\Delta(\sigma\rho)] = \begin{cases}
	1 & \text{if }\rho=\lambda,\\
	0 & \text{otherwise}.
	\end{cases} \]
	The claim follows.
	
	From now on, we may assume that $\lambda$ is not the largest weight in $I_c^{(k)}$. That is, $\lambda=\ell m+i$ where $m=cp^k+m_0$ with $0\le m_0\le p^k-2$. Let $\rho=\ell s+i^*$ where $i^*\in\{i,\bi\}$ and $s=cp^k+s_0$ for some $0\le s_0\le p^k-1$.
	
	We have that $L(\lambda)\cong \bar{L}(m)^F\otimes L(i)$, and by a similar argument to part (a) using \eqref{eqn:1},
	\begin{equation}\label{eqn:4.4b}
	[\Delta(\rho):L(\lambda)] = \begin{cases}
	[\bar{\Delta}(s):\bar{L}(m)] & \text{if } i^*=i,\\
	[\bar{\Delta}(s-1):\bar{L}(m)] & \text{if } i^*=\bi.
	\end{cases}
	\end{equation}
	On the other hand, $\sigma\lambda=\ell\bar{\sigma}(m-1)+\bi$, and since $-1\le m_0-1\le p^k-2$ we have that $\bar{\sigma}(m-1)=cp^k-m_0-1$. As $c\ge 1$ and $m_0\le p^k-2$, we have $\bar{\sigma}(m-1)\ge 1$ and so $T(\sigma\lambda)\cong T(\ell+\bi)\otimes\bar{T}(cp^k-m_0-2)^F$ by \eqref{eqn:5}. Furthermore, $m_0\le p^k-2$ implies $\bar{\sigma}(m)=cp^k-m_0-2$, and hence $T(\sigma\lambda)\cong T(\ell+\bi)\otimes\bar{T}(\bar{\sigma}m)^F$. By \eqref{eqn:2}, the $\Delta$-quotients of $T(\sigma\lambda)$ are precisely $\Delta\big(\ell(u+1)+\bi\big)$ and $\Delta(\ell u+i)$ as $\bar{\Delta}(u)$ varies over the $\bar{\Delta}$-quotients of $\bar{T}(\bar{\sigma}m)$. Hence
	\begin{equation}\label{eqn:4.4b2}
	[T(\sigma\lambda):\Delta(\sigma\rho)] = \begin{cases}
	[\bar{T}(\bar{\sigma}m) : \bar{\Delta}\big(\bar{\sigma}(s-1)-1\big)] & \text{if } i^*=i,\\
	[\bar{T}(\bar{\sigma}m) : \bar{\Delta}\big(\bar{\sigma}(s-1)\big)] & \text{if } i^*=\bi.\\
	\end{cases}
	\end{equation}
	If \emph{either} $i^*=i$ but $\rho$ is not maximal (i.e.~$s_0<p^k-1$) so that $\bar{\sigma}(s)=cp^k-s_0-2=\bar{\sigma}(s-1)-1$, \emph{or} $i^*=\bi$, then the claim follows from the inductive hypothesis combined with \eqref{eqn:4.4b} and \eqref{eqn:4.4b2}.
	
	Finally, suppose $i^*=i$ and $\rho$ is maximal (i.e.~$s_0=p^k-1$). In particular, $\bar{\sigma}(s)=s$ so $\bar{\Delta}(s)$ is simple by Lemma~\ref{lem:c-socle}. Since $\rho\ne\lambda$, we have $m\ne s$ and so $[\bar{\Delta}(s):\bar{L}(m)]=0$. By \eqref{eqn:4.4b}, this implies $[\Delta(\rho):L(\lambda)]=0$.
	On the other hand, \eqref{eqn:4.4b2} gives $[T(\sigma\lambda):\Delta(\sigma\rho)] = [\bar{T}(\bar{\sigma}m) : \bar{\Delta}(n)]$ where $n:=\bar{\sigma}(s-1)-1=(c-1)p^k-1$. Observe that $\bar{\sigma}(n)=n$, so $\bar{\Delta}(n)$ is simple. However, $\bar{\sigma}(m)\ne n$, so $[\bar{T}(\bar{\sigma}m) : \bar{\Delta}(n)]=0$. Therefore $[\Delta(\rho):L(\lambda)]=[T(\sigma\lambda):\Delta(\sigma\rho)]=0$ as desired.
\end{proof}

\bigskip

\subsection{Relating projectives and tilting modules}

Recall for $\mu\in I_d^{(k)}$ with $0\le d<a-1$ that we have $P(\mu)\cong T(\lambda)$ by Corollary~\ref{cor:4.3}, where $\mu=\sigma(\lambda)$ and $\lambda=I_{d+1}^{(k)}$. In particular, such an indecomposable projective $B$--module $P(\mu)$ has a twisted tensor product factorisation by \eqref{eqn:5}, since $T(\lambda)\cong T(\ell+i)\otimes\bar{T}(m-1)^F$ if $\lambda=\ell m+i$. We now consider $P(\mu)$ where $\mu\in I_{a-1}^{(k)}$.

As usual, we proceed by induction on $k$. We will need to keep track of which block our projective modules come from, since the projective cover of a fixed simple module $L(\mu)$ may change when $L(\mu)$ is viewed as a $B_w$--module for different values of $w$. Thus we will write $P_w(\mu)$ for the projective cover of $L(\mu)$ in the block $B_w$. (Indeed, we may also observe from reciprocity $[P(\mu):\Delta(\lambda)]=[\Delta(\lambda):L(\mu)]$ that the number of $\Delta$-quotients of $P_{w'}(\mu)$ may be larger than that of $P_w(\mu)$ if $w'>w$.)

\begin{defn}\label{def:m-in-x}
	Suppose $\lambda\in I_c^{(k)}$ and $\lambda$ is not the largest weight in this interval. Then $\lambda=\ell m+i^*$ for some $cp^k\le m<(c+1)p^k-1$. Since $m\ne (c+1)p^k-1$, there exists $t\in\{0,1,\dotsc,k-1\}$ such that $m\equiv -1$ (mod $p^t$) but $m\not\equiv -1$ (mod $p^{t+1}$).
	As described in Section~\ref{sec:intervals} (ii), we may write $m=(p^t-1)+p^t(cp^{k-t}+m_1)\in\bar{\Delta}(p^t-1)\otimes {\bar{B}}^{F^t}$ where $\bar{B}$ is a primitive block of size $ap^{k-1-t}$, and $m$ corresponds to $cp^{k-t}+m_1\in \bar{I}_c^{(k-1-t)}$. 
	Let $x=k-1-t$ and note that $x\in\{0,1,\dotsc,k-1\}$. We write 
	\[ m\in \bar{I}_c^{(x)} \]
	to denote that $\lambda$ and $m$ have the above form.
\end{defn}

\begin{lem}\label{lem:4.9}
	Suppose $\lambda\in I_c^{(k)}$ and $\lambda$ is not the largest weight in this interval. Let $\lambda=\ell m+i^*$ with $m\in\bar{I}_c^{(x)}$. If $\rho\in I_c^{(k)}$ is such that $[\Delta(\rho):L(\lambda)]\ne 0$, then $\Delta(\rho)$ occurs as a quotient in a $\Delta$-filtration of
	\[ T(\ell+\bistar)\otimes\bar{P}_{(c+1)p^x}(m)^F. \]
\end{lem}

\begin{proof}
	We remark that $T(\ell+\bistar)\otimes\bar{P}_{(c+1)p^x}(m)^F$ indeed belongs to $\mathscr{F}(\Delta)$, since $\bar{P}_{(c+1)p^x}(m)\in\mathscr{F}(\bar{\Delta})$ and then we may identify the $\Delta$-quotients using \eqref{eqn:2}.
	
	Write $\rho=\ell s+j$ with $j\in\{i,\bi\}=\{i^*,\bistar\}$. Since $\lambda=\ell m+i^*\in I_c^{(k)}$ with $c\ge 1$, then $m\ge 1$. Moreover, $[\Delta(\rho):L(\lambda)]\ne 0$ implies $\rho\ge\lambda$, so $s\ge m\ge 1$.
	Then by \eqref{eqn:1}, $\Delta(\rho)=\Delta(\ell s+j)$ has a filtration with quotients $\bar{\Delta}(s-1)^F\otimes L(\bar{j})$ and $\bar{\Delta}(s)^F\otimes L(j)$. Since $L(\lambda)\cong\bar{L}(m)^F\otimes L(i^*)$, we have that
	\[ [\Delta(\rho):L(\lambda)] = \begin{cases}
	[\bar{\Delta}(s):\bar{L}(m)] & \text{if }j=i^*,\\
	[\bar{\Delta}(s-1):\bar{L}(m)] & \text{if }j=\bistar.
	\end{cases} \]
	By reciprocity, this implies
	\[ [\Delta(\rho):L(\lambda)] = \begin{cases}
	[\bar{P}_{(c+1)p^x}(m): \bar{\Delta}(s)] & \text{if }j=i^*,\\
	[\bar{P}_{(c+1)p^x}(m): \bar{\Delta}(s-1)] & \text{if }j=\bistar.
	\end{cases} \]
	We abbreviate $\bar{P}_{(c+1)p^x}(m)$ to $\bar{P}(m)$.
	If $j=i^*$ then $T(\ell+\bistar)\otimes\bar{P}(m)^F$ has a filtration with one of the quotients isomorphic to $T(\ell+\bistar)\otimes\bar{\Delta}(s)^F$, since $T(\ell+\bistar)\otimes (-)^F$ is an exact functor. Since $s\ge 1$, $T(\ell+\bistar)\otimes\bar{\Delta}(s)^F$ has $\Delta(\ell s+i^*)=\Delta(\rho)$ as a $\Delta$-quotient by \eqref{eqn:2}.
	
	If $j=\bistar$ then $[\Delta(\rho):L(\lambda)]\ne 0$ implies that $T(\ell+\bistar)\otimes\bar{P}(m)^F$ has a filtration with one of the quotients isomorphic to $T(\ell+\bistar)\otimes\bar{\Delta}(s-1)^F$. Also by \eqref{eqn:2}, $T(\ell+\bistar)\otimes\bar{\Delta}(s)^F$ has $\Delta(\ell s+i^*)=\Delta(\rho)$ as a $\Delta$-quotient, which concludes the proof.
\end{proof}

\begin{defn}
	For $\gamma\in I_c^{(k)}$, the indecomposable projective module in the block $B_{(c+1)p^k}$ with highest weight $\gamma$ is denoted by $P_{(c+1)p^k}(\gamma)$.
\end{defn}

\begin{prop}\label{prop:new4.8}
	Suppose $\lambda\in I_c^{(k)}$ and $\lambda$ is not the largest weight in this interval. 
	Let $\lambda=\ell m+i^*$ with $m\in\bar{I}_c^{(x)}$. Then $P_{(c+1)p^k}(\lambda)$ has a twisted tensor product factorisation
	\[ P_{(c+1)p^k}(\lambda) \cong T(\ell+\bistar)\otimes\bar{P}_{(c+1)p^x}(m)^F. \] 
\end{prop}

\begin{proof}
	Note when $k=0$, the interval $I_c^{(0)}$ consists of only one weight, so we may assume $k\ge 1$.
	
	\medskip
	
	\noindent\emph{Step 1.} Suppose $k=1$. Then $\lambda=\ell m+i^*$ where $m=cp+m_0$ and $0\le m_0\le p-2$ as $\lambda$ is not maximal in $I_c^{(k)}$. Also $m\in\bar{I}_c^{(0)}$, so we wish to show that $T(\ell+\bistar) \otimes\bar{P}_{c+1}(m)^F\cong P_{(c+1)p}(\lambda)$.
		
	Observe that $T(\ell+\bistar) \otimes\bar{P}_{c+1}(m)^F=T(\ell+\bistar)\otimes\bar{\Delta}(m)^F$ has simple top $L(\lambda)$ by Lemma~\ref{lem:3.3} (or by \eqref{eqn:2}, since it has $\Delta$-quotients $\Delta\big(\ell(m+1)+\bistar\big)$ and $\Delta(\ell m+i^*)=\Delta(\lambda)$).
	Hence there exists a surjective homomorphism $\pi:P_{(c+1)p}(\lambda)\to T(\ell+\bistar)\otimes\bar{\Delta}(m)^F$. 
	
	If $\rho$ is a weight such that $[P_{(c+1)p}(\lambda):\Delta(\rho)]>0$, then $\rho\ge\lambda$. Moreover, $\lambda\in I_c^{(1)}$ and $\rho\in B_{(c+1)p}=I_0^{(1)}\sqcup\cdots\sqcup I_c^{(1)}$, thus $\rho\in I_c^{(1)}$ also. By reciprocity, $[\Delta(\rho):L(\lambda)]>0$, so then by Lemma~\ref{lem:4.9} we find that $\Delta(\rho)$ is also a $\Delta$-quotient of $T(\ell+\bistar)\otimes\bar{\Delta}(m)^F$. Thus $\pi$ is an isomorphism, and $T(\ell+\bistar) \otimes\bar{P}_{c+1}(m)^F\cong P_{(c+1)p}(\lambda)$.

	\medskip
	
	\noindent\emph{Step 2.} We now describe the inductive step. Assume $k\ge 2$, and since $\lambda\in I_c^{(k)}$ is not maximal, $\lambda=\ell m+i^*$ where $m=cp^k+m_0$ and $0\le m_0<p^k-1$. Moreover, $m\in\bar{I}_c^{(x)}$ as in Definition~\ref{def:m-in-x}. We split into two cases depending on whether $m$ is maximal.
	
	\medskip
	
	\noindent\emph{Step 2a.} Suppose $m$ is not maximal. (Specifically, $m$ corresponds to $cp^{k-t}+m_1\in \bar{I}_c^{(k-1-t)}=\bar{I}_c^{(x)}$ and $cp^{k-t}+m_1$ is not the maximal weight of this interval.) Then we may apply the inductive hypothesis.
	By a similar argument to that in Step 1, we deduce from Lemma~\ref{lem:4.9} that 
	\[ T(\ell+\bistar)\otimes\bar{P}_{(c+1)p^x}(m)^F\cong P_{(c+1)p^k}(\lambda). \]
	In particular, the $\Delta$-quotients of $T(\ell+\bistar)\otimes\bar{P}_{(c+1)p^x}(m)^F$ are those of $T(\ell+\bistar)\otimes\bar{\Delta}(u)^F$ as $\bar{\Delta}(u)$ runs over the $\bar{\Delta}$-quotients of $\bar{P}_{(c+1)p^x}(m)$, which are $\Delta\big(\ell(u+1)+\bistar \big)$ and $\Delta(\ell u+i^*)$ by \eqref{eqn:2}.
		
	\medskip
	
	\noindent\emph{Step 2b.} Finally, suppose that $m$ is maximal. Then $\bar{P}_{(c+1)p^x}(m)=\bar{\Delta}(m)$. Observe that $T(\ell+\bistar)\otimes\bar{P}_{(c+1)p^x}(m)^F$ has $\Delta$-quotients $\Delta\big(\ell(m+1)+\bistar \big)$ and $\Delta(\lambda)$, and in particular has simple top $L(\lambda)$. Then $P_{(c+1)p^k}(\lambda)$ surjects onto $T(\ell+\bistar)\otimes\bar{P}_{(c+1)p^x}(m)^F$. By Lemma~\ref{lem:4.9} (since $\lambda$ is not maximal, and there is no restriction on whether $m$ itself is maximal), every $\Delta$-quotient of $P_{(c+1)p^k}(\lambda)$ is also a $\Delta$-quotient of $T(\ell+\bistar)\otimes\bar{P}_{(c+1)p^x}(m)^F$, so in fact $P_{(c+1)p^k}(\lambda)\cong T(\ell+\bistar)\otimes\bar{P}_{(c+1)p^x}(m)^F=T(\ell+\bistar)\otimes\bar{\Delta}(m)^F$. 
\end{proof}

\bigskip


\section{A torsion pair for $B$-$\modcat$}\label{sec:appl}

Fix $B$ a primitive block of size $|B|=ap^k$ where $a\in\{2,3,\dotsc,p\}$ and $k\in\mathbb{N}_0$. As before, the following will be stated for the quantum block; the classical case is similar. In this section, we abbreviate $P(\mu):=P_{ap^k}(\mu)$ and $I_c:=I_c^{(k)}$ when clear from context.

We fix the multiplicity-free projective $B$--module 
\[ G := \bigoplus_{\substack{\mu\in B :\\ P(\mu)\ \text{is tilting}}} P(\mu) \]

Since $P(\mu)$ is a tilting module if and only if it is projective and injective, $G$ is the basic direct sum of all indecomposable projective-injective modules. We define $\cG$, $\cF$, $e$ and $t$ as in Sections~\ref{sec:torsion-prelim} and~\ref{sec:torsion pair}, setting $A=B$. First, we investigate $t(P\mu)$ for $\mu\in B$.

\begin{lem}\label{lem:5.1}
	\begin{itemize}
		\item[(a)] $G = \bigoplus_{\mu\in B\setminus I_{a-1}} P(\mu)$. In particular, $\modtop G = \bigoplus_{\mu\in B\setminus I_{a-1}} L(\mu)$.
		\item[(b)] If $\mu\in I_d$ with $0\le d\le a-2$, then $t(P\mu)=P(\mu)\cong T(\lambda)$ where $\lambda\in I_{d+1}$ and $\mu = \sigma(\lambda)$.
		\item[(c)] Let $\lambda_m$ be the largest weight in $B$. Then $t(P\lambda_m) = L(\sigma\lambda_m)$. This is a tilting module if and only if $a=2$.
		\end{itemize}
\end{lem}

\begin{proof}
	\begin{itemize}
		\item[(a)] From Corollary~\ref{cor:4.3}, if $\mu\in B\setminus I_{a-1}$ then $P(\mu)\cong T(\lambda)$ where $\mu=\sigma\lambda$. Thus if $\mu\in I_{a-1}$ and $P(\mu)$ is a tilting module, then $P(\mu)\cong T(\nu)$ for some $\nu\in I_0$, but this is impossible given the ordering on the weights. Therefore $P(\mu)$ is tilting if and only if $\mu\in B\setminus I_{a-1}$.
		
		\item[(b)] This follows immediately from (a) and the definition of $t$.
		
		\item[(c)] That $t(P\lambda_m) = L(\sigma\lambda_m)$ follows from Lemmas~\ref{lem:3.10} and~\ref{lem:interval-c}. The simple module $L(\sigma\lambda_m)$ is a tilting module precisely when $\sigma\lambda_m$ is the smallest weight in the block (or equivalently, when $\Delta(\sigma\lambda_m)=L(\sigma\lambda_m)$, which follows from Theorem~\ref{thm:simple-delta}). This is equivalent to $\sigma\lambda_m=i\in I_0^{(k)}$, in other words $a=2$, since $\lambda_m\in I_{a-1}$ and $\sigma\lambda_m\in I_{a-2}$.
	\end{itemize}
\end{proof}

\begin{prop}\label{prop:t-tilting}
Let $\lambda\in I_{1}^{(k)}$ and suppose $\lambda$ is not the largest weight in this interval. Let $P(\lambda)=P_{2p^k}(\lambda)$. Then
\begin{itemize}
	\item[(a)] there exists an injective homomorphism $T(\sigma\lambda) \to P(\lambda)$, and
	\item[(b)] $T(\sigma\lambda) = t(P\lambda)$.
\end{itemize}
\end{prop}

\begin{proof}
	Let $\lambda = \ell m + i^*$ where $m=p^k+m_0$ and $0\leq m_0 < p^k-1$. By Proposition \ref{prop:new4.8} we have the factorisation
	\begin{equation}\label{eqn:star}
	P(\lambda) \cong T(\ell + \bar{i^*})\otimes \bar{P}_{2p^x}(m)^F
	\end{equation}
	where $m\in\bar{I}_c^{(x)}$ as in Definition~\ref{def:m-in-x}. Since $0\le m_0\le p^k-2$, we have that $\bar{\sigma}(m)+1=\bar{\sigma}(m-1)$ (see Definition~\ref{def:sigma} (a)).
	
	\begin{itemize}
		\item[(a)] We use induction on $k$, starting with $k=1$ since $\lambda$ is not the largest weight in its interval. Then $m=p+m_0$ where $0\leq m_0\leq p-2$,
		$x=k-1=0$, and $\bar{P}_2(m) = \bar{\Delta}(m)$. By Lemma~\ref{lem:5.1} (c), we have an exact sequence
		\[ 0\to \bar{L}(\bar{\sigma} m) =  \bar{T}(\bar{\sigma} m) \to \bar{P}_2(m). \]
		Applying the (exact) functor $T(\ell+\bistar)\otimes(-)^F$, we obtain
		\[ 0 \to T(\ell + \bar{i^*}) \otimes \bar{T}(\bar{\sigma} m) ^F \to P_{2p}(\lambda). \]
		Now $T(\ell + \bar{i^*})\otimes \bar{T}(\bar{\sigma}m)^F \cong T(\ell(\bar{\sigma}m + 1) + \bar{i^*})$ by \eqref{eqn:5}, but $\sigma\lambda = \ell \bar{\sigma}(m-1) + \bar{i^*}=\ell(\bar\sigma m+1)+\bistar$, so the proof for $k=1$ is concluded.
		
		For the inductive step, with notation as in \eqref{eqn:star}, if $m$ is maximal in $\bar{I}_c^{(x)}$ then we proceed exactly as in the case $k=1$. If $m$ is not maximal, then the inductive hypothesis gives an injective homomorphism $\bar{T}(\bar{\sigma}m)\to \bar{P}_{2p^x}(m)$. As above, we obtain an inclusion of $T(\ell(\bar\sigma m+1)+\bistar)\cong T(\ell+\bistar)\otimes\bar{T}(\bar\sigma m)^F$ into $P_{2p^k}(\lambda)$, so the proof of part (a) is concluded since $\sigma\lambda=\ell(\bar\sigma(m-1))+\bistar$.
		
		\item[(b)] By part (a) and since $T(\sigma\lambda)$ has only composition factors with highest weights in $I_0^{(k)}$, we may identify $T(\sigma\lambda)$ with a submodule of $t(P\lambda)$, since $G=\bigoplus_{\mu\in I_0}P(\mu)$ by Lemma~\ref{lem:5.1}. Let $\mu\in I_0$. We will show that 
		\begin{equation}\label{eqn:5.2}
		[T(\sigma\lambda) : L(\mu)] = [P(\lambda) : L(\mu)].
		\end{equation}
		Assuming \eqref{eqn:5.2}, it follows that the factor module $P(\lambda)/T(\sigma\lambda)$ has only composition factors
		with highest weight in $I_1^{(k)}$. Therefore
		$Be(P(\lambda)/T(\sigma\lambda))=0$, that is, $t(P\lambda) = BeP(\lambda)$ is contained in $BeT(\sigma\lambda) \subseteq  T(\sigma\lambda)$, and so $T(\sigma\lambda)=t(P\lambda)$ as claimed.
		
		To show that \eqref{eqn:5.2} holds, we note that $\sigma:I_1^{(k)}\to I_0^{(k)}$ is a bijection, and so
		\begin{align*}
		[T(\sigma\lambda):L(\mu)] 
		&= \sum_{\rho\in I_1} [T(\sigma\lambda):\Delta(\sigma\rho)]\cdot[\Delta(\sigma\rho):L(\mu)]\\		
		&= \sum_{\rho\in I_1} [\Delta(\rho):L(\lambda)]\cdot[\Delta(\rho):L(\mu)]\\
		&= \sum_{\rho\in I_1} [P(\lambda):\Delta(\rho)]\cdot[\Delta(\rho):L(\mu)] = [P(\lambda):L(\mu)],
		\end{align*}
		where the second equality follows from Proposition~\ref{prop:4.4} and the third by reciprocity.
	\end{itemize}
\end{proof}

\bigskip

\subsection{Ringel self-duality}\label{sec:ringel-self-duality}

From now on, we focus on blocks $B$ of Schur algebras containing $2p^k$ simple modules. 
We choose an orthogonal primitive idempotent decomposition of the identity in $B$, and for each $\lambda\in B$, we take from this a primitive idempotent $e_{\lambda}$ such that $P(\lambda)=Be_{\lambda}$.
Then let $G=Be$ where $e = \sum_{\mu\in I_0} e_\mu$. In preparation for applying Theorem~\ref{thm:2.4} with our choice of $e$ and $t$, we have the following result.

\begin{lem}\label{lem:5.2} 
	Let $a=2$ and $\lambda\in B$. Then $Be\soc P(\lambda)=\soc P(\lambda)$.
\end{lem}

\begin{proof}
	If $\lambda \in I_0$ then $P(\lambda)$ is a tilting module by Corollary~\ref{cor:4.3}. Hence $P(\lambda)$ is a self-dual module with socle $L(\lambda)$. The claim follows since $BeL(\lambda) = L(\lambda)$.
	Now suppose $\lambda\in I_1$ and $\lambda$ is not maximal in $I_1$. By Lemma~\ref{lem:3.3} and the factorisation in Proposition~\ref{prop:new4.8}, $\soc P(\lambda)$ is simple, and so $\soc P(\lambda)=\soc T(\sigma\lambda)$ by Proposition~\ref{prop:t-tilting} (a). But all composition factors of $T(\sigma\lambda)$ have highest weight in $I_0$, so $Be\soc P(\lambda)=\soc P(\lambda)$.
	Finally if $\lambda$ is the largest weight in $I_1$, then $\soc P(\lambda)=\soc \Delta(\lambda)=L(\sigma\lambda)$. The claim follows since $\sigma\lambda\in I_0$.
\end{proof}

Lemma~\ref{lem:5.2} shows that blocks $B$ with $|B|=2p^k$ satisfy the conditions of Theorem~\ref{thm:2.4}, and so we have an injective algebra homomorphism 
\begin{equation}\label{eq:inj}
	t: \End_B(P) \longrightarrow \End_B(t(P))
\end{equation}
where $P=\bigoplus_{\lambda\in B}P(\lambda)$.
Observe that $\End_B(P)$ is the basic algebra of $B$, which in particular is Morita equivalent to $B$, and moreover as quasi-hereditary algebras by \cite{Coul}. 

\begin{thm}
	If $a=2$, then $t(P)=\bigoplus_{\lambda\in B} T(\lambda)$. Moreover, the map $t$ in \eqref{eq:inj} is an isomorphism, and $B$ is Ringel self-dual.
\end{thm}

\begin{proof}
	Let $\lambda_m$ denote the largest weight in $B$. By Lemma~\ref{lem:5.1} (c), we have that $t(P(\lambda_m))=T(i)$ where $i$ is the smallest weight in $B$. 
	For $\lambda\in I_1$ and $\lambda\ne\lambda_m$, we have that $t(P(\lambda))\cong T(\sigma\lambda)$, by Proposition~\ref{prop:t-tilting}. 
	These give all of the indecomposable tilting modules with highest weights in $I_0$, and Corollary~\ref{cor:4.3} shows that $t(P\lambda)$ for $\lambda\in I_0$ give the remaining indecomposable tilting $B$--modules, namely those with highest weights in $I_1$. Hence $t(P)=\bigoplus_{\lambda\in B} T(\lambda)$.
	
	A Ringel dual of $B$ is thus given by $\End_B(t(P))$, so to complete the proof it suffices by \eqref{eq:inj} to show for all $\lambda,\mu\in B$ that we have $\dim\Hom_B(P\lambda,P\mu)=\dim\Hom_B(t(P\lambda),t(P\mu))$, since each component $t_{\lambda\mu}$ is injective (see \eqref{eqn:component}).
	
	This is clear if $\lambda,\mu \in I_0^{(k)}$ since then $P\lambda = t(P\lambda)$ and $P\mu = t(P\mu)$, as $G=\bigoplus_{\gamma\in I_0} P(\gamma)$. 
	If $\lambda\in I_0$ and $\mu\in I_1$, then $P(\lambda)=t(P\lambda)$ and $t(P\mu)=T(\sigma\mu)$ (by Proposition~\ref{prop:t-tilting} if $\mu$ is not maximal in $I_1$, or by Lemma~\ref{lem:5.1} if $\mu$ is maximal). Since $\sigma\mu\in I_0$, $T(\sigma\mu)$ has only composition factors with highest weight in $I_0$. Using that $P(\lambda)$ is projective with top $L(\lambda)$, we see that $\dim\Hom_B(P\lambda,T(\sigma\mu)) = [T(\sigma\mu):L(\lambda)]$, and $\dim\Hom_B(P\lambda,P\mu)=[P(\mu):L(\lambda)]$. By \eqref{eqn:5.2} (with $\lambda$ and $\mu$ interchanged), these two multiplicities are equal. Using duality we get for free the case when $\lambda\in I_1$ and $\mu\in I_0$. Finally, suppose $\lambda,\mu\in I_1$. Recalling that $\sigma:I_1\to I_0$ is a bijection, we have
	\begin{align*}
	\dim\Hom_B(T(\sigma\lambda),T(\sigma\mu)) &= \sum_{\rho \in I_1} [T(\sigma\lambda):\Delta(\sigma\rho)]\cdot [T(\sigma\mu) :\Delta(\sigma\rho)]\\
	&= \sum_{\rho\in I_1} [\Delta\rho:L\lambda] \cdot [\Delta\rho:L\mu]\\
	&= \sum_{\rho\in I_1}[P\lambda : \Delta\rho]\cdot [\Delta\rho: L\mu]\\
	&= [P\lambda : L\mu] = \dim\Hom_B(P\lambda, P\mu)
	\end{align*}
	by Proposition~\ref{prop:4.4} and reciprocity.
\end{proof}

\bigskip



\begin{thebibliography}{99}

\bibitem{ASS} 
{\sc I.~Assem, D.~Simson and A.~Skowro\'{n}ski, }
\newblock \emph{Elements of the Representation Theory of Associative Algebras}, LMS Student Texts, Vol {\bf 65}, Cambridge University Press, 2006. 

\bibitem{Coul}
{\sc K.~Coulembier, }
\newblock The classification of blocks in BGG category $\mathcal{O}$, 
\newblock \emph{Math. Zeit.} {\bf 295} (2020), 821--837. 

\bibitem{C} 
{\sc A.~Cox, }
\newblock \emph{On some applications of infinitesimal methods to quantum groups and related algebras}, Ph.D. thesis, University of London, 1997.

\bibitem{C2}
{\sc A.~Cox,}
\newblock ${\rm Ext}^1$ for Weyl modules for $q$-$GL(2,k)$,
\newblock \emph{Math. Proc. Camb. Phil. Soc.} {\bf 124} (1998), 231--251.


\bibitem{DD} 
{\sc R.~Dipper and S.~Donkin,}
\newblock Quantum $\GL_n$,
\newblock \emph{Proc. London Math. Soc.} (3) \textbf{63} (1991), 165--211.

\bibitem{DJ}
{\sc R.~Dipper and G.~D.~James,}
\newblock The $q$-Schur algebra,
\newblock \emph{Proc. London Math. Soc.} (3) \textbf{59} (1989), 23--50.


\bibitem{Dhom}
{\sc S.~Donkin, }
\newblock Standard homological properties for quantum $\GL_n$,
\newblock \emph{J. Algebra} \textbf{181} (1996), 235--266.

\bibitem{D}
{\sc S.~Donkin, }
\newblock \emph{The $q$-Schur Algebra}, 
\newblock London Mathematical Society Lecture Note Series \textbf{253}, Cambridge University Press, 1998.

\bibitem{E}
{\sc K.~Erdmann, }
\newblock \emph Schur algebras of finite type,
\newblock \emph{Quart. J. Math. Oxford} (2) \textbf{44} (1993), 17--41.

\bibitem{EH}
{\sc K.~Erdmann and A.~Henke,}
\newblock On Ringel duality for Schur algebras, 
\newblock \emph{Math. Proc. Camb. Phil. Soc.} \textbf{132} (2002), 97--116.

\bibitem{EH2}
{\sc K.~Erdmann and A.~Henke,}
\newblock On Schur algebras, Ringel duality and symmetric groups,
\newblock \emph{J. Pure Appl. Algebra} \textbf{169} (2002), 175--199.

\bibitem{EN}
{\sc K.~Erdmann and D.~K.~Nakano,}
\newblock Representation type of $q$-Schur algebras,
\newblock \emph{Trans. Amer. Math. Soc.} \textbf{353} (2001), 4729--4756.

\bibitem{green}
{\sc J.~A.~Green, }
\newblock \emph{Polynomial representations of $\GL_n$}, 
\newblock Lecture Notes in Mathematics \textbf{830}, Springer--Verlag, 1980.

\bibitem{H}
{\sc A.~Henke,}
\newblock The Cartan matrix of the Schur algebra $S(2,r)$,
\newblock \emph{Arch. Math.} \textbf{76} (2001), 416--425.

\bibitem{PW} {\sc B.~Parshall and J.~Wang}, 
\newblock \emph{Quantum linear groups}, 
\newblock Memoirs of the AMS {\bf 439} (1991).


\bibitem{R}
{\sc C.~M.~Ringel,}
\newblock \emph{Tame Algebras and Integral Quadratic Forms}, 
\newblock Lecture Notes in Mathematics \textbf{1099}, Springer, 1985.

\bibitem{R2}{\sc C.~M.~Ringel,}
\newblock The category of modules with good filtrations over a quasi-hereditary algebra has almost split sequences
\newblock \emph{Math. Zeit.} {\bf 208} (1991), 209--225.

\bibitem{X} {\sc S. Xanthopoulos,}
\newblock \emph{On a question of Verma about indecomposable representations of algebraic
groups and their Lie algebras.}
\newblock Ph.D. thesis, University of London, 1992.

\end{thebibliography}
\end{document}